\documentclass[a4paper,11pt]{article}
\usepackage[sc]{mathpazo}
\usepackage[T1]{fontenc}
\usepackage{amsthm,amsmath,amstext,amssymb,amsfonts,hyperref,a4wide}

\newtheorem{theorem}{Theorem}
\newtheorem{lemma}[theorem]{Lemma}

\newtheorem{claim}{Claim}
\newtheorem{proposition}[theorem]{Proposition}
\newenvironment{claimproof}[1]{\par\noindent\emph{Proof.}\space#1}{\hfill $\Diamond$\newline}

\newcommand*{\N}{\mathbb{N}}
\newcommand*{\disc}{\rm disc}
\newcommand*{\herdisc}{\rm{herdisc}}
\newcommand*{\lindisc}{\rm{lindisc}}
\newcommand*{\herlindisc}{\rm{herlindisc}}

\title{On a Ramsey-type problem of Erd\H{o}s and Pach}

\author{Ross J. Kang\thanks{Radboud University Nijmegen. Email: \url{ross.kang@gmail.com}.} \and Eoin Long\thanks{Tel Aviv University. Email: \url{eoinlong@post.tau.ac.il}}\and  Viresh Patel\thanks{University of Amsterdam. Email: \url{viresh.s.patel@gmail.com}. Supported by the Queen Mary - Warwick Strategic Alliance and the Netherlands Organisation for Scientific Research (NWO) through the Gravitation Programme Networks (024.002.003).} \and Guus Regts\thanks{University of Amsterdam. Email: \url{guusregts@gmail.com}. The research leading to these results has received funding from the European Research Council
under the European Union's Seventh Framework Programme (FP7/2007-2013) / ERC grant agreement
n$\mbox{}^{\circ}$ 339109. This author was supported by a NWO Veni grant.} }

\begin{document}
\maketitle
\begin{abstract}
In this paper we show that there exists a constant $C>0$ such that for any graph $G$ on $Ck\ln k$ vertices either $G$ or its complement $\overline{G}$ has an induced subgraph on $k$ vertices with minimum degree at least $\frac12(k-1)$. 
This affirmatively answers a question of Erd\H{o}s and Pach from 1983. 
\begin{footnotesize}

\quad 


\end{footnotesize}
\end{abstract}

\section{Introduction}

Recall that the (diagonal, two-colour) Ramsey number is defined to be the smallest integer $R(k)$ for which any graph on $R(k)$ vertices is guaranteed to contain a homogeneous set of order $k$ --- that is, a set of $k$ vertices corresponding to either a complete or independent subgraph.
The search for better bounds on $R(k)$, particularly asymptotic bounds as $k\to\infty$, is a challenging topic that has long played a central role in combinatorial mathematics (see~\cite{GRS90,Pro13}).

We are interested in a degree-based generalisation of $R(k)$ where, rather than seeking a clique or coclique of order $k$, we seek instead an induced subgraph of order (at least) $k$ with high minimum degree (clique-like graphs) or low maximum degree (coclique-like graphs). Erd\H{o}s and Pach~\cite{ErPa83} introduced this class of problems in 1983 and called them {\em quasi-Ramsey problems}. 
By gradually relaxing the
degree requirement, a spectrum of Ramsey-type problems arise, and Erd\H{o}s and Pach showed that this spectrum exhibits a sharp change in behaviour at a certain point. Naturally, this point corresponds to a degree requirement of half the order of the subgraph sought.
Three of the authors recently revisited this topic together with Pach~\cite{KPPR14+}, and refined our understanding of the threshold for mainly what is referred to in~\cite{KPPR14+} as the {\em variable quasi-Ramsey numbers} (corresponding to the parenthetical `at least' above). In the present paper we focus on the harder version of this problem, the determination of what is called the {\em fixed quasi-Ramsey numbers} (where `exactly' is implicit instead of `at least' above).

Using a result on graph discrepancy, Erd\H{o}s and Pach \cite{ErPa83} proved that there is a constant $C>0$ such that for any graph $G$ on at least $Ck^2$ vertices either $G$ or its complement $\overline{G}$ has an induced subgraph on $k$ vertices with minimum degree at least $\frac12(k-1)$. With an unusual random graph construction, they also showed that the previous statement does not hold with $C'k\ln k/\ln\ln k$ in place of $Ck^2$ for some constant $C'>0$. They asked if it holds instead with $Ck\ln k$. (This was motivated perhaps by the fact that this bound holds for the corresponding variable quasi-Ramsey numbers.)
Our main contribution here is to confirm this, by showing the following.

\begin{theorem}\label{thm:main}
There exists a constant $C>0$ 
such that for any graph $G$ on $Ck\ln k$ vertices, either $G$ or its complement $\overline{G}$ has an induced subgraph on $k$ vertices with minimum degree at least $\frac12 (k-1)$.
\end{theorem}

Although it is short, our proof of Theorem~\ref{thm:main} has a number of different ingredients, including the use of graph discrepancy in Section~\ref{sec:aux}, an application of the celebrated `six standard deviations' result of Spencer \cite{Spen} in Section~\ref{sec:disc} and a greedy algorithm in Section~\ref{sec:greedy} that was inspired by similar procedures for max-cut and min-bisection.
It is interesting to remark that the two discrepancy results we use are of a different nature; the one in Section~\ref{sec:aux} is an anti-concentration result while the result of Spencer is a concentration result.

\section{An auxiliary result via graph discrepancy}\label{sec:aux}

Our first step in proving Theorem~\ref{thm:main} will be to apply the following result. This is a bound on a variable quasi-Ramsey number which is similar to Theorem~3(a) in~\cite{KPPR14+}. The idea of the proof of this auxiliary result is inspired by the sketch argument for Theorem~2 in~\cite{ErPa83}, in spite of the error contained in that sketch (cf.~\cite{KPPR14+}).
\begin{theorem}\label{thm:KPPR}
For any constant $\nu\ge 0$, there exists a constant $C = C(\nu) > 1$ such that for any graph $G$ on $C k\ln k$ vertices, $G$ or its complement $\overline{G}$ has an induced subgraph on $\ell\ge k$ vertices with minimum degree at least $\frac12(\ell-1)+\nu \sqrt{\ell-1}$.
\end{theorem}

Note that the $O(k\ln k)$ quantity is tight up to an $O(\ln\ln k)$ factor by the unusual construction in~\cite{ErPa83} (cf.~also Theorem~4 in~\cite{KPPR14+}). 
The astute reader may later notice that the second-order term $\nu \sqrt{\ell-1}$ in the minimum degree guarantee of Theorem~\ref{thm:KPPR} can be straightforwardly improved to an $\Omega(\sqrt{(\ell-1)\ln\ln \ell})$ term. Since this does not seem to help in our results, we have omitted this improvement to minimise technicalities.
On the other hand, a standard random graph construction yields the following, which certifies that the second-order term cannot be improved to a $\omega(\sqrt{(\ell-1)\ln\ln \ell})$ term.
\begin{proposition}\label{prop:standard}
For any $c>0$, for large enough $k$ there is a graph $G$ with at least $k\ln^c k$ vertices such that the following holds. If $H$ is any induced subgraph of $G$ or $\overline{G}$ on $\ell\ge k$ vertices, then $H$ has minimum degree less than $\frac12(\ell-1) + \sqrt{3c(\ell-1)\ln\ln\ell}$.
\end{proposition}
\begin{proof}
Substitute $\nu(\ell)= \sqrt{(2c\ln\ln\ell)/\ln\ell}$ into the proof of Theorem~3(b) in~\cite{KPPR14+}. (We may not use Theorem~3(b) in~\cite{KPPR14+} directly as stated as it needs $\nu(\ell)$ to be non-decreasing in $\ell$.)
\end{proof}

We use a result on graph discrepancy to prove Theorem~\ref{thm:KPPR}. 
Given a graph $G=(V,E)$, the {\em discrepancy} of a set $X\subseteq V$ is defined as
\[
D(X):= e(X)-\frac{1}{2}\binom{|X|}{2},
\]
where $e(X)$ denotes the number of edges in the subgraph $G[X]$ induced by $X$.
We use the following result of Erd\H{o}s and Spencer~\cite[Ch.~7]{ErSp74}.
\begin{lemma}[Theorem~7.1 of~\cite{ErSp74}]\label{lem:ES}
Provided $n$ is large enough and $t \in \mathbb{N}$ satisfies $\frac12 \log_2 n < t \leq n$, then any graph $G=(V,E)$ of order $n$ satisfies
\[
\max _{S\subseteq V, |S|\le t} |D(S)|\ge \frac{t^{3/2}}{10^3}\sqrt{\ln(5n/t)}.
\]
\end{lemma}

\begin{proof}[Proof of Theorem~\ref{thm:KPPR}]
Let $G=(V,E)$ be any graph on at least $N =\lceil Ck\ln k\rceil$ vertices for a sufficiently large choice of $C$.
We may assume that $k > \frac12 \log_2 N$ because otherwise $G$ or $\overline{G}$ contains a clique of order $k$ by the Erd\H{o}s-Szekeres bound \cite{ErSz35} on ordinary Ramsey numbers.

  For any $X\subseteq V$ and $\nu>0$, we define the following skew form of discrepancy:
\[
D_\nu(X):= |D(X)|-\nu |X|^{3/2}.
\]
 
We now construct a sequence $(H_0,H_1,\dots, H_t)$ of graphs as follows.  Let $H_0$ be $G$ or $\overline{G}$.
At step $i + 1$, we form $H_{i+1}$ from $H_i=(V_i,E_i)$ by letting $X_i\subseteq V_i$ attain the maximum skew discrepancy $D_\nu$ and setting $V_{i+1}:=V_i\setminus X_i$ and $H_{i+1}:=H[V_{i+1}]$.
We stop after step $t+1$ if $|V_{t+1}| < \tfrac12N$.
Let $I^+\subseteq \{1,\dots, t\}$ be the set of indices $i$ for which $D(X_i)>0$.
By symmetry, we may assume 
\begin{align}
\sum_{i\in I^+} |X_i| \ge \frac{1}{4}N.\label{eq:Iplus}
\end{align}

\begin{claim}\label{claim:deg}
For any $i\in I^{+}$ and $x\in X_i$,
$\deg_{H_i}(x)\ge \tfrac12(|X_i|-1) +\nu (|X_i|-1)^{1/2}$.
\end{claim}
\begin{claimproof}
Write $|X_i|=n_i$. We are trivially done if $n_i = 1$, so assume $n_i \ge 2$.
Suppose $x\in X_i$ has strictly smaller degree than claimed and set $X'_i:=X_i\setminus \{x\}$.
Then, since $i\in I^{+}$,
\begin{align*}
D_\nu(X'_i)&\ge e(X'_i)-\frac{1}{2}\binom{n_i-1}{2} -\nu(n_i-1)^{3/2}	
\\
&>e(X_i)-\frac{1}{2}\binom{n_i}{2} -\nu \sqrt{n_i-1}-\nu(n_i-1)^{3/2}.	
\end{align*}
Note that $n_i^{3/2} > n_i^{1/2}+(n_i-1)^{3/2}$,
which by the above  
implies $D_\nu(X'_i)>D_\nu(X_i)$, contradicting the maximality of $D_\nu(X_i)$.
\end{claimproof}

\noindent
Claim~\ref{claim:deg} implies that we may assume for each $i\in I^+$ that $|X_i| \le k-1$, or else we are done.
This gives for any $i_1,\dots,i_4\in I^+$ that
\begin{align}
\left(\sum_{s=1}^4 |X_{i_s}|\right)^{3/2}\le 8(k-1)^{3/2}.\label{eq:4bound}
\end{align}
Writing $I^+=\{i_1,\dots,i_m\}$, we next show the following.
\begin{claim}\label{claim:5/6}
For any $\ell\in \{1,\dots,m-3\}$,
$D(X_{i_{\ell+3}})\le \tfrac56D(X_{i_\ell})$.
\end{claim}

\begin{claimproof}
For $X\subseteq V$, let us write $\nu(X):=\nu |X|^{3/2}$ so that $D_{\nu}(X) = |D(X)| - \nu(X)$.
For $i_1,\dots, i_r\in I^+$, we may write $X_{i_1,\dots, i_r}:=\bigcup_{s=1}^r X_{i_s}$.
For disjoint $X,Y\subseteq V$, we define the \emph{relative discrepancy} between $X$ and $Y$ to be
\[
D(X,Y):=e(X,Y)-\frac{1}{2}|X||Y|,
\]
where $e(X,Y)$ denotes the number of edges between $X$ and $Y$.

Now let $i,j\in I^+$ with $i<j$. 
Then, by the maximality of $D_\nu(X_i)$, we have $D_\nu(X_i \cup X_j) \le D_\nu(X_i)$, i.e.\ 
\[
|D(X_i)+D(X_i,X_j)+D(X_j)|-\nu(X_{i,j})\le |D(X_i)|-\nu(X_{i})= D(X_i)-\nu(X_{i}) ,
\]
and hence 
\begin{align}
D(X_j)\le -D(X_i,X_j)+\nu(X_{i,j}).\label{eq:Dp bound ij}
\end{align}
Applying~\eqref{eq:Dp bound ij} (and the fact that $\nu(X_{i_{\ell+r},i_{\ell+s}})\le \nu(\bigcup_{s=0}^3 X_{i_{\ell+s}})$ for any $r,s\in \{0,1,2,3\}$), we find that
\begin{align}
&D(X_{i_{\ell+1}})+2D(X_{i_{\ell+2}})+3D(X_{i_{\ell+3}})\le -\sum_{0\le r<s\le 3}D(X_{i_{\ell+r}},X_{i_{\ell+s}})+6 \nu(\textstyle \bigcup_{s=0}^3 X_{i_{\ell+s}}).\label{eq:Dp 4bound}
\end{align}
Using $-D(\bigcup_{s=0}^3 X_{i_{\ell+s}})-\nu(\bigcup_{s=0}^3 X_{i_{\ell+s}})\leq D_\nu(\bigcup_{s=0}^3X_{i_{\ell+s}})\le D_\nu(X_{i_\ell})$, we obtain
\begin{align*}
-\sum_{s=0}^3 D(X_{i_{\ell+s}}) -\sum_{0\le r<s\le 3}D(X_{i_{\ell+r}},X_{i_{\ell+s}})
\le D(X_{i_\ell})+\nu(\textstyle \bigcup_{s=0}^3 X_{i_{\ell+s}}),
\end{align*}
which combined with~\eqref{eq:Dp 4bound} implies that
$D(X_{i_{\ell+2}})+2D(X_{i_{\ell+3}})\le 2D(X_{i_\ell})+7\nu(\textstyle \bigcup_{s=0}^3 X_{i_{\ell+s}})$.
From this, we obtain that 
\begin{align}
3D(X_{i_{\ell+3}})\le 2D(X_{i_\ell})+8\nu(\textstyle \bigcup_{s=0}^3 X_{i_{\ell+s}}), \label{eq:disc bound}
\end{align}
where we have used the fact that $D(X_{i_{\ell+3}}) \le D(X_{i_{\ell+2}}) + \nu(\textstyle \bigcup_{s=0}^3 X_{i_{\ell+s}})$, which follows since $D_{\nu}(X_{i_{\ell+3}}) \le D_{\nu}(X_{i_{\ell+2}})$.
Using the fact that the graph $H_{i_{s}}$ for any $s\in \{1,\dots, m\}$ has at least $\frac12N \ge \frac C2 k\ln k$ vertices, it follows by Lemma~\ref{lem:ES} (using our assumption on $k$) that there exists a subset $Y_s\subseteq V_{i_s}$ of size at most $k$ which satisfies 
\[
|D(Y_s)|\ge k^{3/2}\frac{\sqrt{\ln(C\ln k)}}{10^3}.
\]
However, by our choice of $X_{i_s}$, we have
\begin{align*}
D(X_{i_s})
&
\geq D_\nu(X_{i_s})\geq D_\nu(Y_s)\geq |D(Y_s)|-\nu k^{3/2}
\\
& \geq k^{3/2}\left(\frac{\sqrt{\ln(C\ln k)}}{10^3}-\nu\right)\geq 2\left(8\nu\left(\bigcup_{s=0}^3X_{i_{\ell+s}}\right)\right),
\end{align*}
by \eqref{eq:4bound}, provided $C$ is sufficiently large.
Therefore, from~\eqref{eq:disc bound} we find that $3D(X_{i_{\ell+3}})\le 2D(X_{i_\ell})+\tfrac12D(X_{i_\ell})$, proving the claim.
\end{claimproof}

\noindent
Claim~\ref{claim:5/6} now implies that
$(5/6)^{(m-1)/3} D(X_{i_1}) \ge D(X_{i_m})\ge 1$ (assuming for simplicity $m\equiv 1 \pmod 3$), which then implies
\[
m-1\le \frac{3\ln(D(X_{i_1}))}{\ln(6/5)}\le \frac{6}{\ln(6/5)}\ln(k-1).
\]
By~\eqref{eq:Iplus}, we deduce that at least one of the $m$ sets $X_i$ with $i\in I^+$ satisfies
\[
|X_i|\ge \frac{N\ln(6/5)}{25\ln k}.
\]
This last quantity is at least $k$ by a choice of $C$ sufficiently large, contradicting our assumption that $|X_i|\le k-1$ for each $i\in I^+$. This completes the proof.
\end{proof}

\section{Subgraphs of high minimum degree via set-system discrepancy}\label{sec:disc}
In this section we prove, based on a well known discrepancy result of Spencer \cite{Spen}, that from a graph on $\ell=Ck$ vertices with minimum degree at least $\ell/2+C'\sqrt{\ell}$ (with $C'$ depending on $C$) we can select a subgraph on $k$ vertices that has minimum degree at least $k/2$.

We start by defining the various standard notions of discrepancy that we need. 
Suppose $\mathcal{H} = \{A_1, \ldots, A_n\}$ where $A_i \subseteq V = [n]$. 
Let $\chi : V \rightarrow \{-1, 1\}$ be a colouring of $V$ with the 
colours $-1$ and $1$. For any $S \subseteq V$, we write $\chi(S) :=  
\sum_{i \in S} \chi(i)$ and we define the \emph{discrepancy} of $\mathcal{H}$ to be
\[
{\rm disc}(\mathcal{H}) := \min_{ \chi \in \{-1 , 1 \}^V} \max_{S \in \mathcal{H}} \chi(S).
\] 
The result of Spencer \cite{Spen} states that for any such $\mathcal{H}$ we have $\disc(\mathcal{H}) \leq 6\sqrt{n}$.

For $X \subseteq V$, we define $\mathcal{H}|_X:= \{A_1 \cap X ,\ldots , A_n \cap X \}$. Then the \emph{hereditary discrepancy} of $\mathcal{H}$ is defined by
\[
\herdisc(\mathcal{H}) := \max_{X \subseteq V} \disc(\mathcal{H}|_{X}).
\] 
The result of Spencer also immediately implies that $\herdisc(\mathcal{H}) \leq 6\sqrt{n}$ for any $\mathcal{H}$.

Let $A$ be the incidence matrix of $\mathcal{H}$, i.e.~$A$ is the $n \times n$ matrix given by 
\[
A_{ij} =
\begin{cases}
1 &\text{ if } j \in A_i, \\
0 &\text{ otherwise. } 
\end{cases}
\] 
Then we clearly have
\[\disc(\mathcal{H}) = \min_{x \in \{-1,1\}^V} \lVert Ax\rVert_{\infty} =2\min_{x\in \{0,1\}^V}\left\lVert A\left(x - \frac12\mathbb1\right) \right\rVert_{\infty},
\]where $\mathbb 1$ is the all $1$ vector. 

Now we define the \emph{linear discrepancy} by
\begin{equation}\label{eq:lin disc}
\lindisc(\mathcal{H}) := \max_{c \in [0,1]^V} \min_{x \in \{0,1\}^V} \lVert A(x - c) \rVert_{\infty}.
\end{equation}
Note that here we are using $\{0,1\}$-colourings again.
Similarly, we define the hereditary linear discrepancy of $\mathcal{H}$ by
\[
\herlindisc(\mathcal{H}):= \max_{X \subseteq V} \lindisc(\mathcal{H}|_{X}).
\]

A result of Lov{\'a}sz, Spencer, and Vestergombi \cite{LSV} states that $\herlindisc(\mathcal{H}) \leq \herdisc(\mathcal{H})$. (Note that the factor of $2$ from \cite{LSV} is missing to adjust for the slightly different definition we are using.) Combining with Spencer's result, we have 
\[
\lindisc(\mathcal{H}) \leq \herlindisc(\mathcal{H}) \leq \herdisc(\mathcal{H}) \leq 6\sqrt{n}.
\]
If we set $c$ to be the all $p$ vector (for some $p \in [0,1]$) in \eqref{eq:lin disc}, we obtain the following result.

\begin{lemma}\label{lem:cor spen}
Let $A_1, \ldots, A_n \subseteq V = [n]$ and $p \in [0,1]$. Then there exists $Y \subseteq V$ such that, for all $i\in[n]$,
\[
| |A_i \cap Y| - p|A_i| | \leq 6\sqrt{n}.
\] 
\end{lemma}

We use the previous lemma to prove the following result.

\begin{lemma}
\label{lem:thinning exact}
Suppose $G=(V,E)$ is a graph with $\ell = Pk$ vertices for some $P>1$ and $k$ a positive integer, and suppose 
\[
\delta(G) \geq \frac{1}{2} \ell + \eta \sqrt{\ell}
\]
for some $\eta > 0$. Then $G$ has an induced subgraph $H$ on $k$ vertices with minimum degree
\[
\delta(H) \geq \frac{1}{2}k + \left( \frac{\eta}{\sqrt{P}} - 19\sqrt{P} \right) \sqrt{k}.
\]
\end{lemma}
\begin{proof}
Write $V= \{v_1, \ldots, v_{\ell}\}$, let $A_0=V$ and for each $i\in[\ell]$ let $A_i \subseteq V$ be the neighbourhood of $v_i$ in $G$. We apply Lemma~\ref{lem:cor spen} to the sets $A_0, \ldots, A_{\ell-1}$ with $p = (k + 1 + 6\sqrt{\ell})/ \ell$. (Note that if $p>1$ then with a simple calculation it is easy to see we can obtain the desired graph $H$ simply by deleting any $\ell - k$ vertices from $G$.) Thus there exists $Y \subseteq V$ satisfying
\[
| |A_i \cap Y| - p|A_i| | \leq 6\sqrt{\ell}
\] 
for all $i\in\{0, \ldots, \ell-1\}$. Applying this for $i=0$ and noting $A_0 \cap Y = Y$ gives
\[
k+1 = p|A_0|  - 6\sqrt{\ell} \leq 
|Y|   \leq
p|A_0|  + 6\sqrt{\ell} = k + 1 + 12\sqrt{Pk}  
\]
and applying it for $i\in[\ell-1]$ gives
\begin{align*}
|A_i \cap Y| \geq p|A_i| - 6\sqrt{\ell}
\geq \frac{k}{\ell}\left( \frac{1}{2}\ell + \eta \sqrt{\ell} \right) - 6\sqrt{\ell}
&=  \frac{1}{2}k + \eta \frac{k}{\sqrt{\ell}} - 6\sqrt{\ell} \\
&=  \frac{1}{2}k + \left( \frac{\eta}{\sqrt{P}} - 6\sqrt{P} \right) \sqrt{k}.
\end{align*}
Thus $Y$ has between $k+1$ and $k + 1 + 12\sqrt{P}\sqrt{k}$ vertices. Let $Z$ be an arbitrary subset of $Y \setminus \{v_{\ell}\}$ of size $k$ and let $H = G[Z]$. Then since we have removed at most $12\sqrt{Pk} + 1 \leq 13\sqrt{Pk}$ vertices from $Y$ to obtain $Z$, we have for each $i\in[\ell-1]$ that
\begin{align*}
|A_i \cap Z| \geq \frac{1}{2}k + \left( \frac{\eta}{\sqrt{P}} - 19\sqrt{P} \right) \sqrt{k}.
\end{align*}
In particular this means
\[
\delta(H) \geq \frac{1}{2}k + \left( \frac{\eta}{\sqrt{P}} - 19\sqrt{P} \right) \sqrt{k},
\]
as desired.
\end{proof}

\section{Proof of Theorem~\ref{thm:main}}\label{sec:greedy}

To prove the theorem, we use as a subroutine the following algorithm, which is inspired by the greedy algorithm for max-cut or min-bisection.
 \begin{lemma}\label{lem:max cut}
Let $G =(V,E)$ be a graph of order $n$ with $\delta(G)\ge \frac 12(n-1)+t$ for some number $t$.
Let $\alpha\in [0,1]$ and let $a,b\in \N$ such that $a+b=n$.
Then either there exists $A\subseteq V$ of size $a$ such that $\delta(G[A])\ge \frac12a-1+\alpha t$, or there exists $B\subseteq V$ of size $b$ such that $\delta(G[B])\ge \frac12 b-1+(1-\alpha)t$.
\end{lemma}
\begin{proof}
Take any $A\subseteq V$ of size $a$ and let $B:=V\setminus A$.
If there exists $x\in A$ with $\deg_A(x)< \frac12a-1+\alpha t$ and $y\in B$ with $\deg_B(y)<\frac12b-1+(1-\alpha)t$, then move $x$ to $B$ and $y$ to $A$, i.e.~{\em swap} $x$ and $y$.
Note that when there is no such pair of vertices $x,y$ we are done. We just need to prove that, if we keep iterating, then this procedure must stop at some point. 

Consider the number of edges in $G[A]$ before and after we swap $x$ and $y$. The number of edges in $G[A]$ increases by at least \begin{align*}
\deg_A(y) - \deg_A(x) - 1  
&\ge \delta(G) - \deg_B(y) - \deg_A(x) - 1 \ge 1/2,
\end{align*}
(where we subtracted $1$ in case $x$ and $y$ are adjacent).
This shows that we cannot continue to swap pairs indefinitely.
\end{proof}

At last we are ready to prove the main result.
In fact, we prove something stronger.

\begin{theorem}\label{thm:main,threshold}
There exist constants $D,D'>0$ such that for $k \geq 2$ and any graph $G$ on $Dk\ln k$ vertices, $G$ or its complement $\overline{G}$ has an induced subgraph on $k$ vertices with minimum degree 
at least $\frac12 (k-1)+ D'\sqrt{(k-1)/\ln k}$.
\end{theorem}

\begin{proof}
Set  $\nu = 160$, $C=C(\nu)$ as defined according to Theorem~\ref{thm:KPPR}, and $D:=4C$. Also set $D':=1/\sqrt{D}$.

By Theorem~\ref{thm:KPPR}, since $C \cdot 2k \ln( 2k) \le 4C k \ln k = Dk \ln k \le |V(G)|$, we find $G$ or $\overline{G}$ has an induced subgraph 
$H$ on $\ell\ge 2k$ vertices with $\delta(H)\ge \frac12(\ell-1)+\nu\sqrt{\ell-1}$.

Let $x =\ell\bmod k$ (so $x\in \{0,\ldots,k-1\}$).
We can now apply Lemma~\ref{lem:max cut} to $H$ with $a=k+x$, $b=\ell-k-x$, $t = \nu\sqrt{\ell-1}$ and $\alpha=1/2$.
Suppose this gives us a subset $A \subseteq V(H)$ of size $a$ such that 
\[
\delta(H[A]) \ge \frac12a-1+ \frac{1}{2}\nu\sqrt{\ell-1} 
\ge \frac12a + \frac{1}{4}\nu\sqrt{\ell}
\ge \frac12a + \frac{1}{4}\nu\sqrt{a}.
\]
Then $k\le a< 2k$ and, so applying Lemma~\ref{lem:thinning exact} (with $P = a/k \in [1,2]$ and $\eta = \nu/4=40$) yields a subset $A' \subseteq A$ of size $k$ such that 
\[
\delta(H[A'])  \ge 
\frac{1}{2}k + \left( \frac{40}{\sqrt{P}} - 19\sqrt{P} \right) \sqrt{k}
\ge \frac{1}{2}k + \left( \frac{40}{\sqrt{2}} - 19\sqrt{2} \right) \sqrt{k}
\ge \frac12 k + \sqrt{2k},
\] 
which is more than required. 
In case Lemma~\ref{lem:max cut} does not produce such a set $A$, it gives instead a subset $B$ of size $b=\ell-k-x\equiv0\pmod k$  such that $\delta(H[B]) \ge \frac12(b - 1) + \frac{1}{2}\nu\sqrt{\ell-1} - \frac12$. 
We iteratively apply Lemma~\ref{lem:max cut} to $H[B]$ in a binary search to find a desired induced subgraph as follows.

Set $G_0 = H[B]$. Let $\ell_0 := |V(G_0)| = b$ (so that  $k \le \ell_0 \le Dk \ln 2k$ and  $\ell_0 \equiv 0 \pmod k$) and set $t_0 := \frac12 \nu\sqrt{\ell-1} - \frac{1}{2} \ge \frac12 \nu\sqrt{\ell_0-1} - \frac{1}{2}$ (so that $\delta(G_0) \ge \frac{1}{2}(\ell_0 - 1) + t_0$).
 Suppose that $G_i$ is given, where $G_i$ 
has $\ell_i$ vertices with $\ell_i \equiv0\pmod k$ and 
$\delta(G_i) \ge \frac{1}{2}(\ell_i - 1) + t_i$ for some 
number $t_i$. Set $a_i = \lfloor \ell_i/2k \rfloor k$ and 
$b_i = \lceil \ell_i/2k \rceil k$ so that $a_i + b_i = \ell_i$ 
and  $a_i \equiv b_i \equiv 0 \pmod k$. Apply Lemma~\ref
{lem:max cut} with $G=G_i$, $a=a_i$, $b=b_i$, $t=t_i$, and 
$\alpha = \frac{1}{2}$. Then we either obtain a set of vertices $A_i$ of size $a_i$ 
such that $\delta(G_i[A_i]) \ge \frac{1}{2}a_i - 1 + 
\frac{1}{2}t_i$, in which case we set $G_{i+1} := G_i[A_i] = 
H[A_i]$, or we obtain a set of vertices $B_i$ of size $b_i$ such that $\delta(G_i[B_i]) \ge \frac{1}{2}b_i - 1 + \frac{1}{2}t_i$, in which case we set 
$G_{i+1} := G_i[B_i] = H[B_i]$. Now set $\ell_{i+1} = |V(G_{i+1})|$ and note that $\ell_{i+1} \equiv 0 \pmod k$ and $\delta(G_{i+1}) \ge \frac{1}{2}(\ell_{i+1} - 1) + t_{i+1}$, where $t_{i+1} = \frac{1}{2}(t_i - 1)$. Note also that $\ell_{i+1}/k \le \lceil \ell_i/2k \rceil$.

In this way we obtain subgraphs $G_0, G_1, \ldots$ of $G_0=H[B]$ and we see from the recursion for $\ell_i$ above that if $\ell_i > k$ then $\ell_{i+1} < \ell_i$. Thus there exists some $j$ such that $\ell_j = k$ (since $\ell_i \equiv 0 \pmod k$ for all $i$) and an easy computation shows we can assume that $j \le \log_2(\ell_0/k) + 1$. The recursion for $t_i$ implies that $t_i \ge t_02^{-i} - 1$ so that 
\[
t_j \ge \frac{t_0k}{2\ell_0} - 1 
\ge \frac{\nu(\sqrt{\ell_0 - 1} - 1)k}{4 \ell_0} 
\ge \frac{k}{\sqrt{\ell_0}}  
\ge \frac{\sqrt{k}}{\sqrt{D\ln k}}  
= D'\sqrt{\frac{k}{\ln k}} 
\]
(where we used that $t_0 \ge \frac{1}{2}\nu\sqrt{\ell_0 - 1}-\frac12$, that $\ell_0 \geq k \geq 2$ with $\nu = 160$, and that $\ell_0 \le Dk \ln k$). Thus $G_j$ has $k$ vertices and minimum degree at least $\frac{1}{2}(k-1) + D'\sqrt{(k-1)/\ln k}$ and is an induced subgraph of $H[B]$ and hence of $G$ or $\overline{G}$.
\end{proof}

\section{Concluding remarks}

It is tempting to try using the greedy subroutine (Lemma~\ref{lem:max cut}) in a binary search on the output of Theorem~3(a) of \cite{KPPR14+}, but since we cannot control the order of this output graph, the search might require $O(\ln k)$ steps, which would destroy the minimum degree bounds.

Determination of the second-order term in the minimum degree threshold 
for polynomial to super-polynomial growth of the fixed quasi-Ramsey 
numbers is an open problem.
(The corresponding term for the variable quasi-Ramsey numbers was determined in~\cite{KPPR14+}.)
To pose the problem concretely, we recall notation of Erd\H{o}s and Pach. For $c \in [0,1]$ and $k \in \mathbb{N}$, let $R_c^*(k)$ 
be the least number $n$ such that for any graph $G=(V,E)$ on at least 
$n$ vertices, there exists $S \subseteq V$ with $|S|=k$ such that either $
\delta(G[S]) \ge c(k-1)$ or $\delta(\overline{G}[S]) \ge c(k-1)$.
Now consider 
$c = \frac{1}{2} + \varepsilon$ where $\varepsilon = \varepsilon(k)$ is a function of the size $k$ of the subset sought. 
By Theorem~\ref{thm:main,threshold} if $\varepsilon(k) = O(\sqrt{1/
(k-1)\ln k})$ then $R_c^*(k)$ is polynomial in $k$, and by Proposition~\ref{prop:standard} if $\varepsilon(k) = \omega(\sqrt{\ln \ln k/(k-1)})$ then 
$R_c^*(k)$ is superpolynomial in $k$.
Hence the choice of $\varepsilon$ for which we find a transition between polynomial and super-polynomial growth in $k$ of $R_c^*(k)$ is essentially determined to within a $\sqrt{\ln k \ln\ln k}$ factor of $\sqrt{1/(k-1)}$. What is it precisely?

Last, we remark that, in the above notation, our main result is that $R_{1/2}^*(k) \le Ck\ln k$ for some $C>0$, while Erd\H{o}s and Pach showed that $R_{1/2}^*(k) \ge C'k\ln k/\ln \ln k$ for some $C'>0$. They also asked if $R_{1/2}^*(k) \ge C'k\ln k$ for some $C'>0$. This question remains open.

\section*{Acknowledgement}
We thank Noga Alon for stimulating discussions during ICGT 2014 in Grenoble.
We are grateful to Joel Spencer for helpful comments about linear and hereditary discrepancy.
\bibliographystyle{abbrv}
\bibliography{erdospach2016}

\end{document}